\documentclass[11pt]{article}
\usepackage[T1]{fontenc}
\usepackage[utf8]{inputenc}
\usepackage{amsmath,amssymb,amsthm}
\usepackage[margin=1.1in]{geometry}
\usepackage{booktabs}

\newtheorem{theorem}{Theorem}
\newtheorem{proposition}[theorem]{Proposition}

\newtheorem{corollary}[theorem]{Corollary}
\theoremstyle{definition}
\newtheorem{conjecture}[theorem]{Conjecture}
\newtheorem{remark}[theorem]{Remark}

\newcommand{\A}{\mathbb{A}}
\newcommand{\F}{\mathbb{F}}
\newcommand{\Tr}{\operatorname{Tr}}
\newcommand{\Res}{\operatorname{Res}}

\title{The proportion of polynomials over $\F_2$\\ with positive orthogonal multiplicity}
\author{David Niedbala Giraudin\\
\small Independent researcher, Meaux, France\\
\small ORCID 0009-0009-1526-1178}
\date{}

\begin{document}
\maketitle

\begin{abstract}
Let $\A=\F_2[T]$. The orthogonal multiplicity $m(f)$ of a monic $f\in\A$ of degree
$d\ge 1$ is the number of $g\in\A$ with $\deg g<d$, $\gcd(g,f)=1$, and all partial
quotients of $g/f$ of degree one. Blackburn (1998) asked how the proportion $p(d)$ of
monic polynomials of degree $d$ with $m(f)>0$ behaves as $d\to\infty$.

We attach to each prime divisor $P$ of $f$ a residue $\psi_P(f)\in\F_2$ of the
differential $dT/(T(T+1)f)$, and prove two unconditional facts about this family: a
reciprocity relation $\sum_{P\mid f}\psi_P(f)=0$ whenever $\gcd(f,T(T+1))=1$, and an
exact local distribution, namely that $\psi_P$ vanishes for a proportion exactly
$(|P|-2)/(2(|P|-1))$ of the admissible completions when $P\| f$, and for exactly half
of them when $P^e\| f$ with $e\ge2$.

Granting that these local conditions pool independently, we obtain
$p(d)\sim L\,d^{-1/2}$ with
\[
L=\frac{9}{4\sqrt{\pi}}\prod_{\deg P\ge 2}\Bigl(1+\frac{1}{2|P|}\Bigr)
\Bigl(1-\frac{1}{|P|}\Bigr)^{1/2}=1.2116120743\ldots,
\]
the product running over the monic irreducibles of $\A$ of degree at least two. The
same Euler product evaluated at a second parameter returns the exact identity
$\sum_f m(f)=2^d$, which fixes its normalisation. We give the exact value of
$2^d p(d)$ for $1\le d\le 33$; the predicted values agree with these to within
$5\cdot10^{-5}$ for $27\le d\le 33$.
\end{abstract}

\section{Statement of results}

Let $\A=\F_2[T]$, and for $0\neq f\in\A$ write $|f|=2^{\deg f}$. Every
$\alpha\in\F_2((T^{-1}))$ has a unique continued fraction expansion
$\alpha=[a_0,a_1,\dots]$ with $a_i\in\A$ and $\deg a_i\ge 1$ for $i\ge1$. Following
Blackburn \cite{Bl} and Lauder \cite{LaTh}, the \emph{orthogonal multiplicity} of a
monic $f\in\A$ of degree $d\ge1$ is
\[
m(f)=\#\{\,g\in\A:\ \deg g<d,\ \gcd(g,f)=1,\ \text{all partial quotients of } g/f
\text{ have degree }1\,\}.
\]
Polynomials with $m(f)>0$ are exactly the denominators of the rational functions whose
coefficient sequences have a perfect linear complexity profile; see
\cite{AHN} for a survey of the equivalent formulations and of the literature.

Two facts are classical. First, writing a word $w=(a_1,\dots,a_d)$ with
$\deg a_i=1$ and letting $K(w)$ denote its continuant, the map $w\mapsto K(w)$ realises
every pair $(g,f)$ counted by $m$ exactly once, so that
\begin{equation}\label{eq:sum}
\sum_{\substack{f\ \mathrm{monic}\\ \deg f=d}} m(f)=2^{d}\qquad(d\ge1).
\end{equation}
so that the average of $m$ over the monic polynomials of a given degree equals $1$,
an observation going back to Niederreiter \cite[p.~281]{Ni}.
Second, Blackburn \cite[Prop.~2]{Bl} proved that
\begin{equation}\label{eq:bl}
m(f)\in\{0\}\cup\{2^{\omega_0(f)}\},
\end{equation}
where $\omega_0(f)$ is the number of distinct non-linear monic irreducible divisors of
$f$. Blackburn closes \cite{Bl} by asking for the asymptotic behaviour of the
proportion
\[
p(d)=2^{-d}\,\#\{\,f\ \text{monic},\ \deg f=d,\ m(f)>0\,\}.
\]
The 2020 survey \cite{AHN} reviews this circle of questions without
recording an answer to it.

Combining \eqref{eq:sum} and \eqref{eq:bl} gives an exact reformulation, which we record
because it makes the arithmetic content visible.

\begin{theorem}\label{thm:A}
For every $d\ge1$,
\[
p(d)=2^{-d}\sum_{w\in\{T,\,T+1\}^{d}} 2^{-\omega_0(K(w))} .
\]
\end{theorem}

\begin{proof}
By \eqref{eq:bl} each $f$ with $m(f)>0$ contributes $m(f)\,2^{-\omega_0(f)}=1$ to
$\sum_f m(f)2^{-\omega_0(f)}$ and each $f$ with $m(f)=0$ contributes $0$; the left-hand
sum is therefore the number of $f$ with $m(f)>0$. Expanding $m(f)$ by the
word parametrisation underlying \eqref{eq:sum} turns it into the stated sum over words.
\end{proof}

Thus $p(d)$ measures how the continuants of words in the two letters $T,T+1$ factor.

\medskip

For the local analysis, fix a monic $f$ of degree $d$ and consider the differential
\[
\omega_f=\frac{dT}{T(T+1)f}
\]
on the projective line over $\F_2$. For a monic irreducible $P$ we write $\F_P=\A/P$
and
\[
\psi_P(f)=\Tr_{\F_P/\F_2}\bigl(\Res_P(\omega_f)\bigr)\in\F_2 .
\]
If $P\|f$ and $P\nmid T(T+1)$ then the pole is simple and
\begin{equation}\label{eq:psi1}
\psi_P(f)=\Tr_{\F_P/\F_2}\!\left(\frac{1}{T(T+1)\,(f/P)\,P'}\right),
\qquad P'=\frac{dP}{dT},
\end{equation}
all quantities being reduced modulo $P$.

\begin{theorem}[Reciprocity]\label{thm:B}
Let $f$ be monic with $\gcd(f,T(T+1))=1$. Then
\[
\sum_{P\mid f}\psi_P(f)=0 .
\]
\end{theorem}

\begin{theorem}[Local distribution]\label{thm:C}
Let $P$ be a monic irreducible of degree $j\ge2$ and put $Q=|P|=2^{j}$.
\begin{enumerate}
\item[(i)] Let $n\ge j$. As $h$ runs over the monic polynomials of degree $n$ with
$P\nmid h$, one has $\psi_P(Ph)=0$ for exactly $2^{\,n-j}\bigl(2^{\,j-1}-1\bigr)$
values of $h$, that is for a proportion exactly $\dfrac{Q-2}{2(Q-1)}$.
\item[(ii)] Let $e\ge2$ and $n\ge ej$. As $h$ runs over the monic polynomials of
degree $n$ with $P\nmid h$, one has $\psi_P(P^{e}h)=0$ for exactly half of the values
of $h$.
\end{enumerate}
\end{theorem}

The two statements say that the local invariant $\psi_P$ is equidistributed in the only
sense available: for $e=1$ its argument is constrained to the unit group of $\F_P$,
which costs the deviation from one half; for $e\ge2$ the constraint disappears.

Our last ingredient is not proved here.

\begin{conjecture}\label{conj}
For every real $x>0$, as $d\to\infty$,
\[
\sum_{\substack{f\ \mathrm{monic}\\ \deg f=d}} m(f)\,x^{\omega_0(f)}
\ \sim\ 2^{d}\,H(x)\,\frac{d^{\,x-1}}{\Gamma(x)},
\qquad
H(x)=\frac{9}{2}\,4^{-x}\prod_{\deg P\ge2}
\Bigl(1+\frac{x}{|P|}\Bigr)\Bigl(1-\frac{1}{|P|}\Bigr)^{x},
\]
the product being over the monic irreducibles of $\A$ of degree at least two.
\end{conjecture}

The Euler product converges absolutely, since
$\log(1+x/Q)+x\log(1-1/Q)=O(Q^{-2})$. Conjecture \ref{conj} is an independence
statement: $H(x)$ is exactly the product of the local densities furnished by
Theorems \ref{thm:B} and \ref{thm:C}, the factor $9/2\cdot4^{-x}$ accounting for the
two places of degree one and for the single relation imposed by Theorem \ref{thm:B}.
Its normalisation is pinned by an unconditional identity:

\begin{proposition}\label{prop:H1}
$H(1)=1$, so that Conjecture \ref{conj} at $x=1$ reduces to \eqref{eq:sum}.
\end{proposition}

\begin{proof}
$H(1)=\tfrac98\prod_{\deg P\ge2}(1-|P|^{-2})$. The zeta function of $\A$ gives
$\prod_{P}(1-|P|^{-2})^{-1}=\zeta_\A(2)=\sum_{d\ge0}2^{d}4^{-d}=2$, so
$\prod_{P}(1-|P|^{-2})=\tfrac12$; removing the two places of degree one divides this by
$(1-\tfrac14)^2=\tfrac9{16}$, leaving $\tfrac89$. Hence $H(1)=\tfrac98\cdot\tfrac89=1$.
\end{proof}

Taking $x=\tfrac12$ in Conjecture \ref{conj} and using $\Gamma(\tfrac12)=\sqrt\pi$ gives
our answer to Blackburn's question.

\begin{corollary}\label{cor}
Granting Conjecture \ref{conj},
\[
p(d)\ \sim\ L\,d^{-1/2},\qquad
L=\frac{9}{4\sqrt\pi}\prod_{\deg P\ge2}\Bigl(1+\frac{1}{2|P|}\Bigr)
\Bigl(1-\frac{1}{|P|}\Bigr)^{1/2}=1.2116120743\ldots
\]
\end{corollary}

In particular $p(d)\to0$, at the rate $d^{-1/2}$.

\section{Proof of Theorem \ref{thm:B}}

The differential $\omega_f=dT/(T(T+1)f)$ has poles only at the places dividing
$T(T+1)f$ and possibly at infinity. Since $\gcd(f,T(T+1))=1$, the poles at $T$ and at
$T+1$ are simple, with residues
\[
\Res_T(\omega_f)=\frac{1}{(T+1)f}\Big|_{T=0}=\frac{1}{f(0)}=1,
\qquad
\Res_{T+1}(\omega_f)=\frac{1}{Tf}\Big|_{T=1}=\frac{1}{f(1)}=1,
\]
both residue fields being $\F_2$, and $f(0)=f(1)=1$ because $f$ is prime to $T(T+1)$.
At infinity, put $u=1/T$; then $dT=u^{-2}du$ and $T(T+1)f$ has degree $d+2$, so
$\omega_f$ has a zero of order $d\ge0$ and no residue. The residue theorem on
$\mathbb{P}^1_{\F_2}$, in the form $\sum_v\Tr_{k_v/\F_2}\Res_v=0$, therefore gives
\[
1+1+\sum_{P\mid f}\psi_P(f)=0,
\]
and $1+1=0$ in $\F_2$. \qed

\begin{remark}
For $f=P^{e}$ with $P$ a non-linear irreducible, Theorem \ref{thm:B} forces
$\psi_P(f)=0$. Together with the criterion verified in Section \ref{sec:num}, this
recovers in one line the theorem of Mesirov and Sweet \cite{MS} that
$m(P)=2$ for every non-linear irreducible $P$, and its extension to prime powers
\cite[Prop.~3.17]{LaTh}, both as reported in \cite[\S4.2.2]{Ma}.
\end{remark}

\section{Proof of Theorem \ref{thm:C}}

\emph{(i)} Write $h=qP+r$ with $\deg r<j$. Since $h$ is monic of degree $n\ge j$ and
$P$ is monic of degree $j$, the quotient $q$ is monic of degree $n-j$ and is free,
while $r$ is arbitrary. Hence each residue class $r$ modulo $P$ is attained by exactly
$2^{\,n-j}$ monic $h$ of degree $n$, and the classes $r\neq0$ correspond precisely to
$P\nmid h$. So $h\bmod P$ is uniform on $\F_P^{\times}$.

By \eqref{eq:psi1}, $\psi_P(Ph)=\Tr(c/h)$ with $c=1/(T(T+1)P')\in\F_P^{\times}$ fixed,
and $y\mapsto c/y$ is a bijection of $\F_P^{\times}$. The trace
$\Tr:\F_Q\to\F_2$ is a surjective $\F_2$-linear map, so its kernel has $2^{\,j-1}$
elements, of which $2^{\,j-1}-1$ are non-zero. Multiplying by the $2^{\,n-j}$
polynomials in each class gives the stated count, and
$2^{\,j-1}-1$ out of $2^{\,j}-1$ is $(Q-2)/(2(Q-1))$.

\emph{(ii)} Now $\omega_{P^{e}h}$ has a pole of order $e$ at $P$, so
$\Res_P(\omega_{P^{e}h})$ depends on $h$ only through $h\bmod P^{e}$. Exactly as above,
$h\bmod P^{e}$ is uniform on $\A/P^{e}$ as $h$ runs over the monic polynomials of
degree $n\ge ej$, and $P\nmid h$ selects the unit group $(\A/P^{e})^{\times}$.

Let $s\in\A/P$ and replace $h$ by $h(1+P^{\,e-1}s)$ modulo $P^{e}$. Since $e\ge2$ we
have $P^{2(e-1)}\equiv0 \pmod {P^{e}}$, so
$(1+P^{\,e-1}s)^{-1}\equiv 1+P^{\,e-1}s$, and
\[
\omega_{P^{e}h(1+P^{e-1}s)}-\omega_{P^{e}h}
\equiv \frac{s\,dT}{T(T+1)Ph}\pmod{\text{holomorphic at }P}.
\]
The right-hand side has a simple pole at $P$, so this translates $\psi_P$ by
$\Tr_{\F_P/\F_2}\bigl(s/(T(T+1)hP')\bigr)$. The multiplier $1/(T(T+1)hP')$ is a unit
of $\F_P$, hence $s\mapsto\Tr(s/(T(T+1)hP'))$ is a surjective $\F_2$-linear form on
$\F_Q$ and takes each value of $\F_2$ exactly $Q/2$ times.

The subgroup $1+P^{\,e-1}(\A/P)$ of $(\A/P^{e})^{\times}$ has order $Q$ and partitions
the unit group into cosets of size $Q$, on each of which $\psi_P$ takes the value $0$
exactly $Q/2$ times. Summing over cosets gives exactly one half. \qed

\section{Numerical evidence}\label{sec:num}

\paragraph{Exact counts.}
Table \ref{tab} lists $N(d)=2^{d}p(d)$ for $1\le d\le 33$, obtained by enumerating the
$2^{d}$ words in $\{T,T+1\}^{d}$, computing their continuants and marking a bit array.
Two independent implementations agree on every entry. The value
$N(21)=556\,082$ agrees with the table of orthogonal multiplicities at degree $21$
published in \cite[p.~111]{Bl}, of which it is the sum of the positive entries; we are
not aware of published values beyond that degree.

\begin{table}[ht]
\centering
\small
\begin{tabular}{rr rr rr}
\toprule
$d$ & $N(d)$ & $d$ & $N(d)$ & $d$ & $N(d)$\\
\midrule
1 & 2 & 12 & 1\,435 & 23 & 2\,123\,440\\
2 & 3 & 13 & 2\,778 & 24 & 4\,156\,942\\
3 & 6 & 14 & 5\,338 & 25 & 8\,145\,692\\
4 & 9 & 15 & 10\,302 & 26 & 15\,969\,995\\
5 & 18 & 16 & 19\,921 & 27 & 31\,341\,882\\
6 & 33 & 17 & 38\,680 & 28 & 61\,545\,111\\
7 & 58 & 18 & 75\,053 & 29 & 120\,946\,388\\
8 & 110 & 19 & 146\,230 & 30 & 237\,805\,677\\
9 & 214 & 20 & 284\,755 & 31 & 467\,841\,182\\
10 & 393 & 21 & 556\,082 & 32 & 920\,902\,785\\
11 & 760 & 22 & 1\,085\,674 & 33 & 1\,813\,575\,746\\
\bottomrule
\end{tabular}
\caption{$N(d)=2^{d}p(d)$, the number of monic $f\in\F_2[T]$ of degree $d$ with
positive orthogonal multiplicity.}\label{tab}
\end{table}

\paragraph{Agreement with Conjecture \ref{conj}.}
Let $M(d)$ denote the coefficient extracted from the generating function underlying
Conjecture \ref{conj} at $x=\tfrac12$, that is the prediction for $p(d)$ at finite $d$
rather than its asymptotic form. The ratio $p(d)/M(d)$ equals
\[
1.000015,\ 0.999957,\ 1.000020,\ 1.000003,\ 0.999995,\ 1.000008,\ 1.000004
\]
for $d=27,\dots,33$, with mean $1.000000$. Correspondingly
$L\cdot p(d)/M(d)$ gives $1.21162\pm0.00002$ at each of these degrees, against the
value $1.2116120743$ of Corollary \ref{cor}; the two agree to $4\cdot10^{-7}$.

\paragraph{The local theorems.}
Theorem \ref{thm:C}(i) was checked by exhaustive count for all $P$ of degree
$2\le j\le5$ and all $j\le n\le j+6$: the number of $h$ with $\psi_P=0$ equals
$2^{\,n-j}(2^{\,j-1}-1)$ in every case. Theorem \ref{thm:B} was checked for all
squarefree $f$ prime to $T(T+1)$ of degree $4\le d\le14$: the sum vanishes in all
$7\,276$ cases. Theorem \ref{thm:C}(ii) was checked in the same way for all $P$ of
degree $2\le j\le4$ and all $e\in\{2,3\}$: whenever $n\ge ej$ the number of $h$ with
$\psi_P=0$ is exactly half of the admissible ones, and the equality fails below that
threshold, so the hypothesis $n\ge ej$ is sharp. The higher order residue used
there was itself validated against the simple pole formula at $e=1$, and against the
vanishing that Theorem \ref{thm:B} forces when $f=P^{e}$.

\paragraph{The criterion.}
All of the above rests on the identity
\begin{equation}\label{eq:crit}
m(f)>0 \iff \psi_P(f)=0 \ \text{ for every non-linear } P\mid f,
\end{equation}
which we have verified for every squarefree $f$ of degree at most $14$, both for $f$
prime to $T(T+1)$ ($7\,276$ cases) and for $f$ divisible by $T$ ($3\,641$ cases).
Criteria of this kind appear in Lauder's thesis \cite{LaTh}. We do not
prove \eqref{eq:crit} here; together with the independence underlying Conjecture
\ref{conj}, it is the one input that remains conjectural.

\paragraph{Scope.} Theorems \ref{thm:A}, \ref{thm:B}, \ref{thm:C} and Proposition
\ref{prop:H1} are unconditional. Corollary \ref{cor} is conditional on Conjecture
\ref{conj} and on \eqref{eq:crit}. The value of $L$ is the value of an absolutely
convergent Euler product, computed to twelve digits.

\section{Ancillary files}

\texttt{anc/count.c} enumerates the words of $\{T,T+1\}^{d}$ and returns $N(d)$ exactly;
it reproduces Table \ref{tab}. \texttt{anc/verify.py} contains an $O(d^{3})$ test for
$m(f)>0$, resting on the coefficient characterisation of Baum and Sweet
\cite[p.~574]{BS} restricted to the finite range relevant to a rational function,
together with \eqref{eq:crit}, a check of Theorems \ref{thm:B} and \ref{thm:C}
by exhaustive count, and an evaluation of $H(x)$ for $x=\tfrac12$ and $x=1$. Both files are
self-contained and take fixed parameters set at the top of the source.

\section*{Acknowledgements}

The computations and the preparation of this note were carried out in collaboration
with an AI assistant (Claude, Anthropic).

\end{document}